\documentclass[a4paper,10pt]{amsart}
\usepackage[utf8]{inputenc}

\usepackage{graphicx,amssymb,amsmath,amsfonts,amsthm,amscd,hyperref,textcomp}

\newtheorem{proposition}{Proposition}
\newtheorem{lemma}{Lemma}
\newtheorem{cor}{Corollary}
\newtheorem{thm}{Theorem}
\newtheorem{conjecture}{Conjecture}
\newtheorem{rmk}{Remark}\theoremstyle{remark}

\newcommand{\CC}{\mathbb C}

\newcommand{\Fp}{{\mathbb F}_p}
\newcommand{\Fq}{{\mathbb F}_q}

\newcommand{\Ql}{\bar{\mathbb Q}_\ell}

\newcommand{\FF}{\mathcal F}

\newcommand{\LL}{\mathcal L}
\newcommand{\R}{\mathrm R}

\newcommand{\AAA}{\mathbb A}

\newcommand{\HHH}{\mathrm H}

\title{Finite monodromy of some families of exponential sums}
\author{Antonio Rojas-Le\'on}
\address{Dpto. de \'Algebra, Fac. de Matem\'aticas \\
  Universidad de Sevilla \\
  c/Tarfia, s/n \\
  41013 Sevilla, Spain}
\email{arojas@us.es}

\begin{document}

\maketitle

\renewcommand{\thefootnote}{}
\footnote{Partially supported by MTM2016-75027-P (Ministerio de Econom\'{\i}a y Competitividad) and FEDER}

\begin{abstract}
Given a prime $p$ and an integer $d>1$, we give a numerical criterion to decide whether the $\ell$-adic sheaf associated to the one-parameter exponential sums $t\mapsto \sum_x\psi(x^d+tx)$ over $\Fp$ has finite monodromy or not, and work out some explicit cases where this is computable.
\end{abstract}

\section{Introduction}

Consider the affine line $\AAA^1_k$ over a finite field $k$ of characteristic $p>0$. Let $\ell\neq p$ be a prime, and ${\mathcal F}$ al $\ell$-adic sheaf on $\AAA^1_k$ of rank $n$, which can be regarded as a continuous representation
$$
\rho:\pi_1(\AAA^1_k,\bar\eta)\to\mathrm{GL}(n,\Ql)
$$
where $\bar\eta$ is a geometric generic point of $\AAA^1_k$. The arithmetic and geometric monodromy groups $G^{arith}$ and $G^{geom}$ of ${\mathcal F}$ are defined to be the Zariski closures of the images of $\rho$ (resp. of its subgroup $\pi_1(\AAA^1_{\bar k},\bar\eta)$). According to \cite{deligne1980conjecture} (see \cite[Chapter 3]{katz1988gauss} for a more explicit statement), under certain conditions (which are usually fulfilled after taking a Tate twist of $\mathcal F$), these groups govern the distribution of the Frobenius traces of the sheaf $\mathcal F$: more precisely, if $G^{geom}=G^{arith}$, the Frobenius traces are equidistributed as the traces of random elements of a maximal compact subgroup of $G^{geom}$ as $\# k$ grows.

These groups also determine the asymptotic values of the higher moments associated to the trace function of $\mathcal F$, which are related to the dimension of the invariant subspaces of certain tensor powers of the given representation of $G^{geom}$ via $\rho$. See \cite{katz2005mma} for a detailed exposition of the topic.

In this article we will be concerned with a special class of sheaves, which are a subset of the class of so-called Airy sheaves, smooth sheaves on $\AAA^1_k$ of rank $n$ with a single slope $\frac{n+1}{n}$ at infinity, which can also be characterized as the Fourier transform of smooth sheaves of rank $1$ with slope $>1$ at infinity. The monodromy of these sheaves was extensively studied by O. Šuch in \cite{such2000monodromy}, who gave a full classification of their possible non-finite monodromy groups \cite[Propositions 11.6, 11.7]{such2000monodromy}.

Let $d\geq 2$ be a prime to $p$ integer. Let $k=\Fq$ and let $\psi:k\to\CC$ be the additive character given by $\psi(t)=\exp(2\pi it/p)$. Let $[d]:\AAA^1_k\to\AAA^1_k$ be the $d$-th power map, and $\LL_{\psi(t^d)}=[d]^\ast\LL_\psi$ the pull-back of the Artin-Schreier sheaf $\LL_\psi$ on $\AAA^1_k$ associated to $\psi$. It is a smooth sheaf on $\AAA^1_k$ of rank $1$, with slope $d$ at infinity. Its Fourier transform ${\mathcal F}_d$ is then a smooth Airy sheaf on $\AAA^1_k$ of rank $d-1$ with a single slope $\frac{d}{d-1}$ at infinity. The Frobenius trace of ${\mathcal F}_d$ at a point $t\in k$ is given (up to sign) by
$$
\sum_{x\in k}\psi(x^d+tx).
$$

 The main goal of this article is giving a numerical criterion to determine whether the geometric monodromy (and therefore the arithmetic one after a suitable Tate twist) of ${\mathcal F}_d$ is finite. This is done in Proposition \ref{numcriterion}, and some specific cases are worked out explicitly in section  4. Moreover we show that, in the case where the monodromy is not finite, the given representation of the monodromy group is Lie irreducible, which allows to completely determine the arithmetic and geometric monodromy groups via the results in \cite{such2000monodromy}. 
 
 The most important case for applications is $p=2$. In that case, we show that the monodromy of ${\mathcal F}_d$ is finite for $d$ of the form $2^a+1$ or $\frac{2^a+1}{2^b+1}$, and we conjecture that these are the only cases where the monodromy is finite. This case is important for its relation with \emph {almost perfect nonlinear} functions. If the monodromy is not finite, then both $G^{geom}$ and $G^{arith}$ are the full symplectic group $\mathrm{Sp}(d-1,{\mathbb C})$. In particular, the (arithmetic) fourth moment of the trace function of ${\mathcal F}_d$ is 3. But this fourth moment can be computed explicitly, and it is equal to the number of (absolute) irreducible components of the polynomial $x^d+y^d+z^d+(x+y+z)^d\in k[x,y,z]$ minus one. So, if ${\mathcal F}_d$ does not have finite monodromy, the polynomial
 $$
 \frac{x^d+y^d+z^d+(x+y+z)^d}{(x+y)(x+z)(y+z)}
 $$
 is absolutely irreducible. By \cite{aubry} this implies that the function $f(x)=x^d$ is not almost perfect nonlinear on ${\mathbb F}_{2^n}$ for any sufficiently large $n$. This gives a new algebro-geometric approach to the study of such problems.
 
 The author would like to thank Daqing Wan for bringing this problem to his attention.

\section{A numerical criterion for the finiteness of the monodromy of ${\mathcal F}_d$}

Let $k=\Fp$ and ${\mathcal F}_d$ be as in the introduction, with $d\geq 3$ prime to $p$. We want to determine the values of $(p,d)$ such that ${\mathcal F}_d$ has finite geometric monodromy.

\begin{lemma}
 The determinant of the Tate-twisted sheaf ${\mathcal F}_d(1)$ is geometrically trivial and arithmetically of finite order.
\end{lemma}

\begin{proof}
 The determinant $\det {\mathcal F}_d(1)$ is a smooth rank one sheaf on $\AAA^1_k$, which is geometrically trivial by \cite[Theorem 17]{katz-monodromy}. So it is of the form $\alpha^{deg}$ for some $\ell$-adic unit $\alpha$. Then $\det {\mathcal F}_d(1)$ will be arithmetically of finite order over $\AAA^1_{k}$ if and only if $\alpha$ is a root of unity. In order to prove this, we will explicitely evaluate the action of Frobenius on $\det {\mathcal F}_d(1)$ at $t=0$, which is equal to $\alpha$. By replacing $k$ with a finite extension if necessary, we may assume that $d|q-1$, where $q=\# k$.

 Let $r\geq 1$, and let $k_r$ be the extension of $k$ of degree $r$ inside a fixed algebraic closure $\bar k$. The trace of the action of Frobenius on ${\mathcal F}_d(1)$ at $t=0\in k_r$ is given by
$$
\frac{1}{q^{r/2}}\sum_{x\in k_r}\psi_r(x^d)
$$
where $\psi_r(x)=\psi(\mathrm{Tr}_{k_r/\Fp}(x))$. We have
$$
S_r:=\sum_{x\in k_r}\psi_r(x^d)=\sum_{u\in k_r}\psi_r(u)\cdot\#\{x^d=u\}
=\sum_{u\in k_r}\psi_r(u)\sum_{\chi^d={\mathbf 1}}\chi(u)
$$
where the inner sum is taken over the set of multiplicative characters of $k_r$ with trivial $d$-th power. Since $d|q-1$, every such character is obtained, by composition with the norm map, from one such multiplicative character of $k$. Then we have
$$
\sum_{u\in k_r}\psi_r(u)\sum_{\chi^d={\mathbf 1}}\chi(u)
=\sum_{\chi^d={\mathbf 1}}\sum_{u\in k_r}\psi_r(u)\chi(u)=-\sum_{\chi\neq{\mathbf 1};\chi^d={\mathbf 1}}G(\psi_r,\chi)
$$
where
$$
G(\psi_r,\chi)=-\sum_{u\in k'_r}\psi_r(u)\chi(u)
$$
is the Gauss sum associated to $\chi$. Using the Hasse-Davenport relation $G(\psi_r,\chi)=G(\psi_1,\chi)^r$, we deduce
$$
\exp\sum_{r\geq 1} S_r\frac{T^r}{r}=\exp\left(-\sum_{r\geq 1}\sum_{\chi\neq{\mathbf 1};\chi^d={\mathbf 1}}\frac{(G(\psi_1,\chi)T)^r}{r}\right)=
$$
$$
=\prod_{\chi\neq{\mathbf 1};\chi^d={\mathbf 1}}\exp\left(-\sum_{r\geq 1}\frac{(G(\psi_1,\chi)T)^r}{r}\right)=\prod_{\chi\neq{\mathbf 1};\chi^d={\mathbf 1}}(1-G(\psi_1,\chi)T)
$$

So the Frobenius eigenvalues at $t=0\in k$ are $G(\psi_1,\chi)$ for the $d-1$ non-trivial multiplicative characters $\chi$ of $k$ such that $\chi^d={\mathbf 1}$. The determinant is then the product of these Gauss sums. Using the well-known relation $G(\psi_1,\chi)G(\psi_1,\chi^{-1})=\chi(-1)q$ we see that this product is $\pm q^{(d-1)/2}$ if $d$ is odd, or $\pm q^{(d-2)/2}G(\psi_1,\rho)$ if $d$ is even, where in the latter case $\rho$ denotes the unique order $2$ multiplicative character. Since $G(\psi_1,\rho)^2=G(\psi_1,\rho)G(\psi_1,\rho^{-1})=\rho(-1)q$, in both cases the product is $q^{(d-1)/2}$ times a root of unity. So Frobenius acts on $\det {\mathcal F}_d(1)=(\det {\mathcal F}_d)(d-1)$ by multiplication by a root of unity.
\end{proof}

The following corollary is simply a restatement of \cite[Theorem 8.14.4]{katz1990esa}:

\begin{cor}
 The sheaf ${\mathcal F}_d(1)$ has finite arithmetic monodromy if and only if it has finite geometric monodromy, if and only if for every finite extension $k_r$ of $k$ and every $t\in k_r$, the trace of the action of Frobenius on ${\mathcal F}_d(1)$ at $t$ is an algebraic integer.
\end{cor}

The last condition is equivalent to the trace of the action of Frobenius on ${\mathcal F}_d$ being a multiple of $\sqrt{q}$ as an algebraic integer. Let us spell out what this means explicitely:

\begin{cor}
 The sheaf ${\mathcal F}_d$ has finite geometric monodromy if and only if for every $r\geq 1$ and every $t\in k_r$, $\sum_{x\in k_r}\psi_r(x^d+tx)$ is divisible by $p^{r/2}$ as an algebraic integer.
\end{cor}

For our purposes we will need the following equivalent statement:

\begin{proposition}\label{crit1}
 The sheaf ${\mathcal F}_d$ has finite geometric monodromy if and only if for every $r\geq 1$ the sum $\sum_{x\in k_r}\psi_r(x^d)$ is divisible by $p^{r/2}$ as an algebraic integer and, for every non-trivial multiplicative character $\chi:k_r^\times\to\CC^\times$, $G_r(\chi)\cdot \sum_{x\in k_r^\times}\psi_r(x^d)\bar\chi(x)$ is divisible by $p^{r/2}$ as an algebraic integer. It is sufficient that the condition holds for every $r$ which is a multiple of a certain $r_0\geq 1$.
\end{proposition}

\begin{proof}
 This is an explicit version of \cite[Theorem 8.14.6]{katz1990esa}. Suppose that for every $r\geq 1$ and every $t\in k_r$, $\sum_{x\in k_r}\psi_r(x^d+tx)$ is divisible by $p^{r/2}$ as an algebraic integer. Then, in particular, $\sum_{x\in k_r}\psi_r(x^d)$ is divisible by $p^{r/2}$. Furthermore, for every non-trivial multiplicative character $\chi:k_r^\times\to\CC^\times$, the sum
$$
\sum_{t\in k_r^\times}\chi(t)\sum_{x\in k_r}\psi_r(x^d+tx)=\sum_{x\in k_r}\psi_r(x^d)\sum_{t\in k_r^\times}\chi(t)\psi_r(tx)=-G_r(\chi)\cdot \sum_{x\in k_r^\times}\psi_r(x^d)\bar\chi(x)
$$
is also divisible by $p^{r/2}$.

 Conversely, if $\sum_{x\in k_r^\times}\psi_r(x^d)$ is divisible by $p^{r/2}$, so is 
$$\sum_{t\in k^\times}\sum_{x\in k_r^\times}\psi_r(x^d+tx)=\sum_{t\in k}\sum_{x\in k_r^\times}\psi_r(x^d+tx)-\sum_{x\in k_r^\times}\psi_r(x^d)=p^r-\sum_{x\in k_r^\times}\psi_r(x^d)
$$

Since $\sum_{t\in k_r^\times}\chi(t)\sum_{x\in k_r}\psi_r(x^d+tx)$ is also divisible by $p^{r/2}$ for every non-trivial $\chi:k_r^\times\to\CC^\times$, by Fourier inversion $\sum_{x\in k_r}\psi_r(x^d+tx)$ is divisible by $p^{r/2}$ for every $t\in k_r^\times$.

The last statement is a consequence of the fact that having finite geometric monodromy is invariant under extension of scalars to a finite extension of the base field.
\end{proof}

\begin{lemma}
 Let $z\in k_r^\times$ and $\chi:k_r^\times\to\CC^\times$ be a multiplicative character. Then
$$
\sum_{x|x^d=z}\chi(x)=\sum_{\eta|\eta^d=\chi}\eta(z).
$$
\end{lemma}

\begin{proof}
 Let $z_0$ and $\chi_0$ be generators of the multiplicative cyclic groups $k^\times$ and $\widehat{k^\times}$ respectively. Let $z=z_0^a$ and $\chi=\chi_0^b$ with $0\leq a,b\leq p^r-2$, and let $s=\gcd(d,p^r-1)$. Then the equation $x^d=z$ (respectively $\eta^d=\chi$) has solutions if and only if $s|a$ (resp. $s|b$), in which case it has exactly $s$ solutions, given by $x_0z_0^{(p^r-1)i/s}$ (resp. $\eta_0\chi_0^{(p^r-1)i/s}$) for $i=0,1,\ldots,s-1$, where $x_0$ (resp. $\eta_0$) is a particular solution.

If $s\not |a$ and $s\not | b$ then both sides of the equation are zero. Suppose that $s|a$ and $s\not | b$. Then
$$
\sum_{x|x^d=z}\chi(x)=\sum_{i=0}^{s-1}\chi(x_0z_0^{(p^r-1)i/s})=\chi(x_0)\sum_{i=0}^{s-1}\left(\chi_0(z_0)^{(p^r-1)b/s}\right)^i
$$
Since $\chi_0(z_0)$ is a primitive $(p^r-1)$-th root of unity and $s$ does not divide $b$, $\chi_0(z_0)^{(p^r-1)b/s}\neq 1$, so
$$
\chi(x_0)\sum_{i=0}^{s-1}\left(\chi_0(z_0)^{(p^r-1)b/s}\right)^i=\chi(x_0)\frac{\chi_0(z_0)^{(p^r-1)b}-1}{\chi_0(z_0)^{(p^r-1)b/s}-1}=0
$$
and the equality is true in this case. Dually, it is also true if $s|a$ and $s\not |b$.

Suppose now that $s|a$ and $s|b$. Then the left hand side of the equation is
$$
\sum_{x|x^d=z}\chi(x)=\chi(x_0)\sum_{i=0}^{s-1}\left(\chi_0(z_0)^{(p^r-1)b/s}\right)^i=s\cdot\chi(x_0)
$$
and the right hand side is
$$
\sum_{\eta|\eta^d=\chi}\eta(z)=\sum_{i=0}^{s-1}\eta_0(z)\chi_0^{(p^r-1)i/s}(z)=\eta_0(z)\sum_{i=0}^{s-1}\left(\chi_0(z_0)^{(p^r-1)a/s}\right)^i=s\cdot\eta_0(z)
$$
We conclude by noticing that
$$
\chi(x_0)=\eta_0^d(x_0)=\eta_0(x_0^d)=\eta_0(z).
$$
\end{proof}

Using this lemma, we get
$$
\sum_{x\in k_r^\times}\psi_r(x^d)\bar\chi(x)=\sum_{z\in k_r^\times}\psi_r(z)\sum_{x^d=z}\bar\chi(x)=\sum_{z\in k_r^\times}\psi_r(z)\sum_{\eta^d=\chi}\bar\eta(z)=-\sum_{\eta^d=\chi}G_r(\bar\eta)
$$
and
$$
\sum_{x\in k_r}\psi_r(x^d)=1+\sum_{z\in k_r^\times}\psi_r(z)\#\{x|x^d=z\}=1+\sum_{z\in k_r^\times}\psi_r(z)\sum_{\eta^d={\mathbf 1}}\eta(z)=
$$
$$
=1+\sum_{\eta^d={\mathbf 1}}\sum_{z\in k_r^\times}\psi_r(z)\eta(z)=-\sum_{\eta\neq{\mathbf 1},\eta^d={\mathbf 1}}G_r(\eta)
$$
This allows to give yet another criterion for finite monodromy:
\begin{proposition}\label{prop2}
 The sheaf ${\mathcal F}_d$ has finite geometric monodromy if and only if for every $r\geq 1$ and every non-trivial multiplicative character $\eta:k_r^\times\to\CC^\times$, the Gauss sum $G_r(\eta)$ is divisible by $p^{r/2}$ if $\eta^d$ is trivial, and the product $G_r(\eta)G_r(\bar\eta^d)$ is divisible by $p^{r/2}$ is $\eta^d$ is non-trivial. It is sufficient that the condition holds for every $r$ which is a multiple of a certain $r_0\geq 1$.
\end{proposition}

\begin{proof}
 By proposition \ref{crit1} and the previous remark, if these products of Gauss sums are divisible by $p^{r/2}$ then ${\mathcal F}_d$ has finite monodromy. Conversely, suppose that ${\mathcal F}_d$ has finite monodromy. Then for every $r\geq 1$ and every non-trivial $\chi:k_r^\times\to\CC^\times$, the sums $A_r=\sum_{\eta\neq{\mathbf 1},\eta^d={\mathbf 1}}G_r(\eta)$ and $B_r(\chi)=\sum_{\eta^d=\chi}G_r(\chi)G_r(\bar\eta)$ are divisible by $p^{r/2}$ as algebraic integers. We need to show that the individual summands are also divisible by $p^{r/2}$. By the Hasse-Davenport relation, by passing to a finite extension of $k_r$ we may assume that $d|p^r-1$. Then for every $m\geq 1$ there are either $0$ (in which case there is nothing to prove) or $d$ characters $\eta$ of $k_{rm}^\times$ such that $\eta^d=\chi$, which are obtained from those of $k_r^\times$ by composition with the norm map. 
  By the Hasse-Davenport relation we have that $A_{rs}=\pm\sum_{\eta\neq{\mathbf 1},\eta^d={\mathbf 1}}G_r(\eta)^s$ and $B_{rs}(\chi)=\sum_{\eta^d=\chi}G_r(\chi)^sG_r(\bar\eta)^s$ are divisible by $p^{rs/2}$ as algebraic integers. The result is then a consequence of the following lemma.
\end{proof}

\begin{lemma}
 Let $\alpha_1,\ldots,\alpha_d$ be algebraic integers such that $\alpha_1^s+\cdots+\alpha_d^s$ is divisible by $p^{s/2}$ for every $s\geq 1$. Then $\alpha_i$ is divisible by $p^{1/2}$ for every $i=1,\ldots,d$. 
\end{lemma}

\begin{proof}
 This is a well known result, see eg. \cite{ax1964}. Let $K$ be the completion of $\mathbb Q(\alpha_1,\ldots,\alpha_d)$ on a prime over $p$. Since $\alpha_1^s+\cdots+\alpha_d^s$ is divisible by $p^{s/2}$, the power series
 $$
 g(T):=\sum_{s=0}^\infty(\alpha_1^s+\cdots+\alpha_d^s)T^s
 $$
 converges for $|T|_p < p^{1/2}$, so all its poles $x$ must have $|x|_p\geq p^{1/2}$, that is, $(p^{1/2}x)^{-1}$ are $p$-adic integers. But the poles are $\alpha_1^{-1},\ldots,\alpha_d^{-1}$, since
 $$
 g(T)=\frac{1}{1-\alpha_1 T}+\cdots+\frac{1}{1-\alpha_d T}.
 $$
 So $(p^{1/2}\alpha_i^{-1})^{-1}=p^{-1/2}\alpha_i$ is an algebraic integer for all $i=1,\ldots,d$, that is, $\alpha_i$ is divisible by $p^{1/2}$.
\end{proof}

Finally we can make this more explicit thanks to Stickelberger's theorem. Fix $r\geq 1$. For every integer $1\leq x\leq p^r-1$ let $[x]_{p,r}$ be the sum of the $p$-adic digits of $x$. It is an integer between $1$ and $r(p-1)$: for instance, $[1]_{p,r}=[p]_{p,r}=1$, and $[p^r-1]_{p,r}=r(p-1)$. If $x$ is an arbitrary integer, we define $[x]_{p,r}:=[y]_{p,r}$, where $1\leq y\leq p^r-1$ is the unique integer such that $x\equiv y$ (mod $p^r-1$).

It is easy to see from this definition that $[px]_{p,r}=[x]_{p,r}$ and $[-x]_{p,r}=r(p-1)-[x]_{p,r}$ for every $x\in{\mathbb Z}$ which is not a multiple of $p^r-1$. If $x$ is not a multiple of $p^r-1$ we have the following well-known explicit formula for $[x]_{p,r}$, where $\{x\}$ denotes the fractional part of a real number $x$:
$$
[x]_{p,r}=(p-1)\sum_{i=0}^{r-1}\left\{\frac{p^ix}{p^r-1}\right\}.
$$

\begin{thm}\label{numcriterion}
 The sheaf ${\mathcal F}_d$ has finite geometric monodromy if and only if for every $r\geq 1$ and every integer $1\leq x\leq p^r-2$, we have
$$
[dx]_{p,r}\leq [x]_{p,r}+\frac{r(p-1)}{2}.
$$
It is sufficient that the condition holds for every $r$ which is a multiple of a certain $r_0\geq 1$.
\end{thm}

\begin{proof}
 Fix $r\geq 1$. The Gauss sums on $k_r^\times$ take values on the finite extension of $\mathbb Q$ generated by the $p(p^r-1)$-th roots of unity. By the Stickelberger theorem, if $\omega$ denotes the Teichmüller character of $k_r^\times$ (which generates the character group), the $p$-adic valuation of the Gauss sum associated to $\omega^{j}$ for $1\leq j\leq p^r-2$ is given by $\frac{1}{p-1}[j]_{p,r}$. 
 
 Applying this to the criterion of Proposition \ref{prop2}, we get that ${\mathcal F}_d$ has finite monodromy if and only if for every $1\leq j\leq p^r-2$ we have $[j]_{p,r}\geq\frac{r(p-1)}{2}$ if $dj$ is divisible by $p^r-1$, and $[j]_{p,r}+[-dj]_{p,r}\geq\frac{r(p-1)}{2}$ otherwise.
 
 If $dj$ is divisible by $p^r-1$ this can be rewritten as
 $$
 [dj]_{p,r}=r(p-1)\leq [j]_{p,r}+\frac{r(p-1)}{2}
 $$
 and, if $dj$ is not a multiple of $p-1$, it is equivalent to
 $$
 [dj]_{p,r}=r(p-1)-[-dj]_{p,r}\leq r(p-1)+[j]_{p,r}-\frac{r(p-1)}{2}=[j]_{p,r}+\frac{r(p-1)}{2}.
 $$
\end{proof}

For computational purposes, it is convenient to have a sufficient condition for the monodromy of ${\mathcal F}_d$ to be finite. We have the following criterion, which is a generalization of \cite[Lemma 13.5]{katz2007g2}.

\begin{proposition}
 For $r\geq 1$ an integer, let $f_r:[0,1]\to{\mathbb R}$ be the piecewise linear function defined by
 $$
 f_r(x)=\sum_{i=0}^{r-1}\{p^ix\}+\sum_{i=0}^{r-1}\{-dp^ix\}.
 $$
 Suppose that for some integer $r_0\geq 1$ we have
 \begin{enumerate}
  \item $f_{r_0}(\frac{a}{d})\geq\frac{r_0}{2}$ for $a=1,\ldots,d-1$
  \item $\lim_{x\to\frac{a}{p^{r_0-1}d}^{-}}f_{r_0}(x)\geq\frac{r_0}{2}$ for $a=1,\ldots,p^{r_0-1}d$
 \end{enumerate}
 Then the monodromy of ${\mathcal F}_d$ is finite.
\end{proposition}

\begin{proof}
 Since the function $f_r$ is piecewise linear with constant negative slope and its points of discontinuity are $\frac{a}{2^{r-1}d}$ for $a=1,\ldots,p^{r-1}d$, the two conditions imply that $f_{r_0}(x)\geq\frac{r_0}{2}$ for every $x\in {\mathbb Z}_{(p)}\cap (0,1)$.
 
Let $r$ be a multiple of $r_0$, and $1\leq x\leq p^r-1$ an integer. Then
$$
f_r\left(\frac{x}{p^r-1}\right)
=f_{r_0}\left(\frac{x}{p^r-1}\right)+f_{r_0}\left(\frac{p^{r_0}x}{p^r-1}\right)+\cdots+f_{r_0}\left(\frac{p^{(r/r_0-1)r_0}x}{p^r-1}\right)\geq
$$
$$
\geq \frac{r}{r_0}\frac{r_0}{2}=\frac{r}{2}.
$$

Then, if $dx$ is a multiple of $p^r-1$,
$$
[x]_{p,r}=(p-1)\sum_{i=0}^{r-1}\left\{\frac{p^ix}{p^r-1}\right\}=(p-1)f_r\left(\frac{x}{p^r-1}\right)\geq(p-1)\frac{r}{2}.
$$
and otherwise,
$$
[x]_{p,r}+[-dx]_{p,r}=(p-1)f_r\left(\frac{x}{p^r-1}\right)\geq(p-1)\frac{r}{2}.
$$
so ${\mathcal F}_d$ has finite monodromy by Theorem \ref{numcriterion}.
\end{proof}

For instance, for $p=2$, $d=5$ satisfies the condition for $r=4$, as one can easily check.

\section{Explicit results}

In this section we will use Theorem \ref{numcriterion} to give some explicit results. First of all, we can recover the known fact \cite[Proposition 5]{katz-monodromy} that, if $p>2d-1\geq 5$, then the monodromy is not finite.

\begin{cor}
 Suppose that $p\geq 2d+1\geq 7$. Then ${\mathcal F}_d$ does not have finite monodromy.
\end{cor}

\begin{proof}
 Let $p=qd+r$ with $q\geq 2$ and $1\leq r\leq d$. We claim that $[dq]_{p,1}>[q]_{p,1}+\frac{p-1}{2}$, so ${\mathcal F}_d$ can not have finite monodromy by Theorem \ref{numcriterion}.

Since $q<dq\leq p-1$, $[dq]_{p,1}=dq$ and $[q]_{p,1}=q$. So the inequality is equivalent to $2(d-1)q>p-1=qd+r-1$ or, equivalently, $q(d-2)>r-1$. Now
$$
q(d-2)\geq 2(d-2)=d+d-4\geq d-1\geq r-1 
$$
with equality if and only if $d=r=3,q=2$, in which case $p=9$ is not prime.
\end{proof}

\begin{lemma} \label{sum}
 For every $r\geq 1$ and $x,y\in{\mathbb Z}$ we have $[x+y]_{p,r}\leq[x]_{p,r}+[y]_{p,r}$.
\end{lemma}

\begin{proof}
 It suffices to prove it for $1\leq x,y\leq p^r-1$. First of all, it is clear that, if an integer $z\geq 1$ can be written as a sum of $m$ powers of $p$, then $[z]_{p,r}\leq m$ for every $r\geq 1$. Conversely, if $1\leq z\leq p^r-1$, then $z$ can be written as a sum of $[z]_{p,r}$ powers of $p$. 
 
 So $x$ (resp. $y$) can be written as a sum of $[x]_{p,r}$ (resp. $[y]_{p,r}$ powers of $p$), and therefore $x+y$ can be written as a sum of $[x]_{p,r}+[y]_{p,r}$ powers of $p$. We conclude that $[x+y]_{p,r}\leq [x]_{p,r}+[y]_{p,r}$.
\end{proof}

\begin{cor}\label{case1}
 Let $d=p^a+1$ for some integer $a\geq 1$. Then ${\mathcal F}_d$ has finite monodromy.
\end{cor}

\begin{proof}
 We need to show that $[dx]_{p,r}\leq [x]_{p,r}+\frac{r(p-1)}{2}$ for every $r\geq 1$ and every $1\leq x\leq p^r-2$. If $[x]_{p,r}\geq\frac{r(p-1)}{2}$ this is obvious, since $[dx]_{p,r}\leq r(p-1)$. Suppose that $[x]_{p,r}\leq\frac{r(p-1)}{2}$. Then
$$
[dx]_{p,r}=[(p^a+1)x]_{p,r}\leq[p^ax]_{p,r}+[x]_{p,r}=2[x]_{p,r}\leq [x]_{p,r}+\frac{r(p-1)}{2}
$$
for every $1\leq x\leq p^r-2$.
\end{proof}

\begin{cor}
 Let $d=\frac{p^{a}+1}{p^b+1}$, with $a>b\geq 1$. Then ${\mathcal F}_d$ has finite monodromy.
\end{cor}

\begin{proof}
By lemma \ref{sum}, $[(p^b+1)z]_{p,r}\leq 2[z]_{p,r}$ for every $z\in{\mathbb Z}$. Taking $z=-dx$ and using that $[-x]_{p,r}=r(p-1)-[x]_{p,r}$, we get
$$
r(p-1)-[(p^a+1)x]_{p,r}=r(p-1)-[(p^b+1)dx]_{p,r}=[-(p^b-1)dx]_{p,r}\leq 
$$
$$\leq 2[-dx]_{p,r}= 2r(p-1)-2[dx]_{p,r} \Rightarrow 
$$
$$
\Rightarrow [dx]_{p,r}\leq \frac{r(p-1)}{2}+\frac{1}{2}[(p^a+1)x]_{p,r} \leq [x]_{p,r}+\frac{r(p-1)}{2}.
$$
\end{proof}

\begin{rmk} In the situation of the previous corollary, $a$ must be of the form $bc$ with $c$ odd. Indeed, let $a=bc+r$ with $0\leq r < b$. Since $p^a+1$ is a multiple of $p^b+1$, we have
 $$
0\equiv p^a+1 = p^{bc}p^r+1\equiv (-1)^cp^r+1 (\mbox{mod }p^b+1).
$$
Since $|(-1)^cp^r+1|<|p^b+1|$, we conclude that $(-1)^cp^r=-1$, that is, $r=0$ and $c$ is odd.
\end{rmk}

In the case $p=2$ we conjecture that the only cases where the monodromy is finite are the ones covered in the previous corollaries. This has been checked to be true computationally for $d$ up to 10000.

\begin{conjecture}
 Let $p=2$. Then ${\mathcal F}_d$ has finite monodromy if and only if $d$ has the form $2^a+1$ for some $a\geq 1$ or $\frac{2^a+1}{2^b+1}$ for some $b\geq 1$ and $a=bc$ with odd $c\geq 3$.
\end{conjecture}

\section{The monodromy in the non-finite case}

In this section we will completely determine the geometric monodromy group of $\FF_d$ in the case where it is infinite. By \cite[Proposition 11.1]{such2000monodromy}, if the monodromy is not finite, then $\FF_d$ is either Lie-irreducible or Artin-Schreier induced. We will see that the latter case is not possible.

\begin{proposition}\label{lie-irreducible}
 Suppose that the monodromy of $\FF_d$ is not finite. Then $\FF_d$ is Lie-irreducible.
\end{proposition}

\begin{proof}
 Suppose that $\FF_d$ were Artin-Schreier induced. Then the proof of \cite[Proposition 11.1]{such2000monodromy} shows that $\FF_d\otimes\widehat\FF_d$ contains an Artin-Schreier subsheaf $\LL_{\psi(at)}$ for some $a\in{\bar k}^\ast$. That is, 
$$
\mathrm{Hom}(\FF_d\otimes\widehat\FF_d,\LL_{\psi(at)})\neq 0
$$
for some $a\in{\bar k}^\ast$ or, equivalently,
$$
\HHH^0(\AAA^1_{\bar k},\FF_d\otimes\widehat\FF_d\otimes\LL_{\psi(at)})\neq 0
$$
for some $a\in{\bar k}^\ast$. By Poincaré duality, this is equivalent to 
$$
\HHH^2_c(\AAA^1_{\bar k},\FF_d\otimes\widehat\FF_d\otimes\LL_{\psi(at)})\neq 0
$$
for some $a\in{\bar k}^\ast$, since the dual of $\LL_{\psi(at)}$ is $\LL_{\psi(-at)}$.

Given that $\FF_d=\R^1\pi_!\LL_{\psi(x^d+tx)}$ and $\R^i\pi_!\LL_{\psi(x^d+tx)}=0$ for $i\neq 1$, where $\pi:\AAA^2_{\bar k}\to\AAA^1_{\bar k}$ is the projection $(x,t)\mapsto t$, the K\"unneth formula gives
$$
\FF_d\otimes\widehat\FF_d\cong\R^2\sigma_!(\LL_{\psi(x^d+tx)}\otimes\LL_{\psi(-y^d-ty)})\cong\R^2\sigma_!(\LL_{\psi(x^d-y^d+t(x-y))})
$$
and
$$
\R^i\sigma_!(\LL_{\psi(x^d+tx)}\otimes\LL_{\psi(-y^d-ty)})=0 \mbox{ for }i\neq 2
$$
where $\sigma:\AAA^2_{\bar k}\times_{\AAA^1_{\bar k}}\AAA^2_{\bar k}\cong\AAA^3_{\bar k}\to\AAA^1_{\bar k}$ is the projection $(x,y,t)\mapsto t$. Therefore
$$
\HHH^2_c(\AAA^1_{\bar k},\FF_d\otimes\widehat\FF_d\otimes\LL_{\psi(at)})\cong\HHH^2_c(\AAA^1_{\bar k},\R^2\sigma_!(\LL_{\psi(x^d-y^d+t(x-y))})\otimes\LL_{\psi(at)})\cong
$$
$$
\cong\HHH^4_c(\AAA^3_{\bar k},\LL_{\psi(x^d-y^d+t(x-y+a))})
$$
by the projection formula and the Leray spectral sequence. If $\tau:\AAA^3_{\bar k}\to\AAA^2_{\bar k}$ denotes the projection $(x,y,t)\mapsto (x,y)$, the sheaf $\R^i\tau_!\LL_{\psi(t(x-y+a))}$ is supported on the line $y=x+a$, on which it is constant of rank $1$ for $i=2$ and vanishes for $i\neq 2$, by looking a the fibers. So
$$
\R^i\tau_!\LL_{\psi(x^d-y^d+t(x-y+a))})\cong\LL_{\psi(x^d-y^d)}\otimes\R^i\tau_!\LL_{\psi(t(x-y+a))}
$$
is isomorphic to $\LL_{\psi(x^d-y^d)}$ on the line $y=x+a$ for $i=2$, and zero otherwise. So we get
$$
\HHH^4_c(\AAA^3_{\bar k},\LL_{\psi(x^d-y^d+t(x-y+a))})\cong\HHH^2_c(\AAA^2_{\bar k},\LL_{\psi(x^d-y^d)}\otimes\R^i\tau_!\LL_{\psi(t(x-y+a))})\cong
$$
$$
\cong
\HHH^2_c(\{y=x+a\},\LL_{\psi(x^d-y^d)})\cong\HHH^2_c(\AAA^1_{\bar k},\LL_{\psi(x^d-(x+a)^d)})
$$
via the isomorphism $\{y=x+a\}\to \AAA^1_{\bar k}$ given by the projection on the $x$-coordinate. We conclude that
$$
\mathrm{Hom}(\FF_d\otimes\widehat\FF_d,\LL_{\psi(at)})\neq 0
$$
for some $a\in{\bar k}^\ast$ if and only if
$$
\HHH^2_c(\AAA^1_{\bar k},\LL_{\psi(x^d-(x+a)^d)})\neq 0
$$
for some $a\in{\bar k}^\ast$. But given a polynomial $f(x)\in\bar k[x]$, $\HHH^2_c(\AAA^1_{\bar k},\LL_{\psi(f(x))})\neq 0$ if and only if the sheaf $\LL_{\psi(f(x))}$ is constant, which happens only when $f(x)$ is Artin-Schrier equivalent to a constant, that is $f(x)=g(x)^p-g(x)$ for some $g(x)\in\bar k[x]$.

Since $x^d-(x+a)^d=-\sum_{i=1}^d{d\choose i}a^ix^{d-i}$ has degree $\leq d-1$, such a polynomial $g$ would have degree $\leq\frac{d-1}{p}$. In particular, every monomial with non-zero coefficient in $f$ of degree $>\frac{d-1}{p}$ would have degree a multiple of $p$. In other words, for every $i<d-\frac{d-1}{p}$ such that $d-i$ is not a multiple of $p$, the binomial coefficient ${d\choose i}$ would be a multiple of $p$.

We already know that for $d$ of the form $p^r+1$ the monodromy of $\FF_d$ is finite, so the result is then a consequence of the following lemma
\end{proof}

\begin{lemma}
 Suppose that $d$ is not of the form $p^r+1$ for some $r\geq 1$. Then there exists some positive integer $l<d-\frac{d-1}{p}$ such that $d-l$ is not a multiple of $p$ and ${d\choose l}$ is not a multiple of $p$. 
\end{lemma}

\begin{proof}
 Let $r=\mathrm{ord}_p(d-1)$. We will see that $l=p^r$ satisfies the stated conditions. Since $d-1$ is not a power of $p$, we have $d-1\geq 2p^r=2l$, so
$$
l\leq \frac{d-1}{2}=d-\frac{d+1}{2}< d-\frac{d-1}{2}\leq d-\frac{d-1}{p}.
$$
Also, $d-l$ is not a multiple of $p$: if $r=0$ then $d-l=d-1$ is not a multiple of $p$ by definition of $r$. If $r>0$ then $d-l=(d-1)-l+1$ with $d-1$ and $l$ multiples of $p$, so $d$ is not a multiple of $p$. It remains to check that ${d\choose l}$ is not a multiple of $p$. That is, that $\mathrm{ord}_p(d(d-1)\cdots(d-l+1))=\mathrm{ord}_p((d-1)\cdots(d-l+1))=\mathrm{ord}_p(l!)$.

In fact, we will check that $\mathrm{ord}_p(d-1-j)=\mathrm{ord}_p(l-j)$ for every $j=0,1,\ldots,l-2$. For $j=0$ it is clear by definition of $l$. For $j\geq 1$, since $j<l=p^r$, we have $\mathrm{ord}_p(j)<r$, so $\mathrm{ord}_p(l-j)=\mathrm{ord}_p(j)=\mathrm{ord}_p(d-1-j)$.
\end{proof}

Using the results of Katz and Such, this allows to completely determine the geometric monodromy groups in the non-finite case

\begin{cor}
 Let $G$ be the geometric monodromy group of $\FF_d$. If $G$ is not finite, then
\begin{enumerate}
 \item If $p=2$, then $G=\mathrm{Sp}_{d-1}$.
 \item If $p\neq 2$ and $d$ is odd, then $G=\mathrm{Sp}_{d-1}$.
 \item If $p\neq 2$ and $d$ is even, then $G=\mathrm{SL}_{d-1}$.
\end{enumerate}
\end{cor}

\begin{proof}
 By Proposition \ref{lie-irreducible}, $G$ is Lie-irreducible in its given representation. If $p=2$, then $\FF_d$ is self-dual (as it has real Frobenius traces), so by \cite[Proposition 11.7]{such2000monodromy}, the Lie algebra of $G$ is either $\mathfrak{sp}_{d-1}$ in its standard representation or $\mathfrak{e}_7$ in its 56-dimensional representation. But for $d=57=\frac{2^9+1}{2^3+1}$ the monodromy of $\FF_d$ is finite by Corollary \ref{case1}, so we must be in the former case. So the identity component $G^0$ of $G$ must be $\mathrm{Sp}_{d-1}$, and therefore $G=\mathrm{Sp}_{d-1}$ by self-duality of $\FF_d$.

 Suppose now that $p\neq 2$. By \cite[Proposition 11.6]{such2000monodromy}, $G^0$ is then either $\mathrm{Sp}_{d-1}$ or $\mathrm{SL}_{d-1}$ in their standard representations. If $d$ is even the former case is not possible, so $G^0=\mathrm{SL}_{d-1}$ and $G=\det^{-1}(\det(G))$. But by \cite[Theorem 17]{katz-monodromy}, the determinant of $\FF_d$ is geometrically trivial, so $G=\mathrm{SL}_{d-1}$. If $d$ is odd, then $\FF_d$ is again self-dual (as it has real Frobenius traces), so by the previous argument $G$ must be $\mathrm{Sp}_{d-1}$.
\end{proof}

\bibliographystyle{amsalpha}
\bibliography{bibliography}

\providecommand{\bysame}{\leavevmode\hbox to3em{\hrulefill}\thinspace}
\providecommand{\MR}{\relax\ifhmode\unskip\space\fi MR }
\providecommand{\MRhref}[2]{%
  \href{http://www.ams.org/mathscinet-getitem?mr=#1}{#2}
}
\providecommand{\href}[2]{#2}
\begin{thebibliography}{AMR10}

\bibitem[AMR10]{aubry}
Y.~Aubry, G.~McGuire, and F.~Rodier, \emph{A few more functions that are not
  apn infinitely often}, Contemporary Math. \textbf{518} (2010), 23--31.

\bibitem[Ax64]{ax1964}
James Ax, \emph{Zeroes of polynomials over finite fields}, American Journal of
  Mathematics \textbf{86} (1964), no.~2, 255--261.

\bibitem[Del80]{deligne1980conjecture}
Pierre Deligne, \emph{La conjecture de {W}eil. {II}}, Publications
  Math{\'e}matiques de l'IH{\'E}S \textbf{52} (1980), no.~1, 137--252.

\bibitem[Kat87]{katz-monodromy}
Nicholas~M. Katz, \emph{On the monodromy groups attached to certain families of
  exponential sums}, Duke Mathematical Journal \textbf{54} (1987).

\bibitem[Kat88]{katz1988gauss}
\bysame, \emph{{G}auss sums, {K}loosterman sums, and monodromy groups}, Annals
  of Mathematics Studies, vol. 116, Princeton University Press, 1988.

\bibitem[Kat90]{katz1990esa}
\bysame, \emph{Exponential sums and differential equations}, Annals of
  Mathematics Studies, vol. 124, Princeton University Press, 1990.

\bibitem[Kat05]{katz2005mma}
\bysame, \emph{Moments, monodromy, and perversity: A diophantine perspective},
  Princeton University Press, 2005.

\bibitem[Kat07]{katz2007g2}
\bysame, \emph{$g_2$ and hypergeometric sheaves}, Finite Fields and Their
  Applications \textbf{13} (2007), no.~2, 175--223.

\bibitem[{\v S}uc00]{such2000monodromy}
O.~{\v S}uch, \emph{Monodromy of {A}iry and {K}loosterman sheaves}, Duke
  Mathematical Journal \textbf{103} (2000), no.~3, 397--444.

\end{thebibliography}

\end{document}